\documentclass[12pt]{article}   	
\usepackage[T1]{fontenc}
\usepackage[latin9]{inputenc}
\usepackage{geometry}                		
\usepackage{graphicx}				
								
\usepackage{amsthm}
\usepackage{amsmath}
\usepackage{authblk,array,setspace} 
\usepackage{amssymb,natbib}
\usepackage[pdftex,colorlinks=true,linkcolor=blue,citecolor=blue,urlcolor=blue,bookmarks=false,pdfpagemode=None]{hyperref}

\theoremstyle{plain}
\newtheorem{thm}{Theorem}
\theoremstyle{plain}

\theoremstyle{plain}

  \theoremstyle{plain}
  \newtheorem{lem}[thm]{Lemma}
  \theoremstyle{plain}
  \theoremstyle{plain}
  
  \theoremstyle{plain}

\newcommand{\tU}{\tilde{U}}
\newcommand{\tV}{\tilde{V}}

\begin{document} 
\title{Sharp Total Variation Bounds for Finitely Exchangeable Arrays\protect\thanks{Alexander M. Volfovsky is an NSF Postdoctoral Fellow in the Department of Statistics at Harvard University (volfovsky@fas.harvard.edu). Edoardo M.~Airoldi is an Associate Professor of Statistics at Harvard (airoldi@fas.harvard.edu).}} 
\author{Alexander Volfovsky}
\author{Edoardo M. Airoldi}
\affil{Department of Statistics, Harvard University}
\date{}

\maketitle

\begin{abstract}
In this article we demonstrate the relationship between finitely exchangeable arrays and finitely exchangeable sequences. We 
then derive sharp bounds on the total variation distance between distributions of finitely and infinitely exchangeable arrays.\newline

\textbf{Keywords:} Finite exchangeability, arrays, total variation bound, networks, hypothesis testing 

\end{abstract}

\newpage
\tableofcontents

\newpage
\section{Introduction}
The study of invariance properties of probability distributions
has long been the focus of statisticians and probabilists \citep{de1756doctrine,de1820theorie,de1931funzione,
hewitt1955symmetric,aldous1985exchangeability,austin2013hierarchical}. 
The most 
common type of invariance is the notion of independent and
identically distributed distributions. It is fundamental for
classical results such as the weak and strong laws of large
numbers and central limit theorems. Relaxation 
of the statistical notion of independence in this definition 
to simply requiring that the order of a sequence of random 
variables not affect their joint distribution gives rise
to the notion of exchangeable random variable sequences. These variables are marginally
identically distributed but they can exhibit different forms of dependence.
Infinite exchangeability of a sequence of random variables
 is in turn related back to independent and
identically distributed random variable sequences via a famous
theorem due to de Finetti \citep{de1931funzione,de1972probability}. 
Considering an infinitely exchangeable sequence
$X_1,X_2,\dots$ of binary random variables the theorem states that there is a unique 
probability measure $\mu$ such that for all $n\geq 1$
\begin{equation}
\Pr(X_1=e_1,\dots,X_n=e_n)=\int p^{\sum e_i}(1-p)^{n-\sum e_i}d\mu(p).
\label{eqn:defin}
\end{equation}
Generalizations of this famous theorem to higher 
dimensional arrays were given independently
by \citet{hoover1989tail} and \citet{aldous1981representations}: For example 
consider an infinitely row-column weakly exchangeable binary matrix $X$.
 The distribution of such a matrix is invariant under the permutation of the row
 and column indices of the matrix by the same permutation. Intuitively, such a 
 matrix can represent relational data, where the shared index set
 of the rows and columns represents an actor, while the entries of the
 matrix represent the existence of a relationship among the actors.  
The Aldous-Hoover representation for this matrix is 
$$X_{ij}={\bf{1}}[W(U_i,U_j)\geq\lambda_{ij}],$$ where
$U_i$ and $\lambda_{ij}$ are 
independent uniform$(0,1)$ random variables and $W$
is a measurable symmetric function from $[0,1]^2\rightarrow[0,1]$. 
Intuitively, the $U_i$ can be interpreted as actor attributes
while the $\lambda_{ij}$ as pairwise attributes.

These results have served as foundation for many Bayesian methods.
For example, proofs of the consistency of Bayesian procedures rely
on the assumption that the data is infinitely exchangeable, that is
conditionally i.i.d \citep{berk1970consistency}. Similarly, the analysis of relational data heavily
relies on the Aldous-Hoover representation for arrays \citep{hoff2002latent}.
The power of these results is derived from the assumption
that the index set of the array is infinite. Specifically, it is well known that de Finetti's theorem does not
hold for finitely exchangeable sequences \citep{diaconis1977finite,diaconis1980finite}.
A similar result has been known, though apparently
not written down, for the failure of the Aldous-Hoover
representation to hold for finitely row-column exchangeable
arrays. Consider the following distribution
on $2\times 2$ binary matrices: $P(x_{11}=x_{22}=0)=1$,
$P(x_{12}=x_{21})=0$ and $P(x_{12}=1)=P(x_{21}=1)=1/2$. 
This is clearly a jointly row-column exchangeable distribution
but an Aldous-Hoover representation would require:
\begin{align*}
0&=P(x_{12}=0,x_{21}=0)=\int (1-p)^2d\mu(p)\\
0&=P(x_{12}=1,x_{21}=1)=\int p^2 d\mu(p)
\end{align*}
for $\mu$ a mixing measure as in the de Finetti representation in Eq~\eqref{eqn:defin}. 


Notions of finite invariance described above have
received a lot of attention in the case of sequences, but much less so
for networks and general arrays. 
Specifically, our work is motivated by the pervasiveness of the 
assumption of joint row-column exchangeability
throughout the field of network analysis \citep{holland1983stochastic,hoff2002latent,diaconis2007graph,airoldi2009mixed,bickel2009nonparametric}. 
While certain network data have the potential 
representation as a sample from an infinite population,
it is very common to observe the whole finite network such as 
in the study of trade between countries \citep{volfovsky2013testing}, 
friendship networks among students in the same class \citep{hoff2013likelihoods}
and interactions among monks in a monastery \citep{sampson1969crisis}. In these
cases an assumption of exchangeability of the nodes is still desirable as there
is no information in the individual labels, but the assumption of infinite exchangeability
is inappropriate since there is no infinite network the
data could have been sampled from.  

This article is motivated by these finitely exchangeable
network scenarios and is not intended as a review of 
infinite exchangeability. A comprehensive overview
of the assumptions and representations 
of infinite exchangeability for networks and general 
arrays 
is available in \citet{orbanz2013bayesian}.
In the next section we derive
sharp bounds on the total variation distance between 
finitely exchangeable network distributions and their
infinitely exchangeable extensions. This is done
by relating the distribution of the 
singular value decomposition of a relational
matrix to the distribution of finitely exchangeable sequences. 
Section~\ref{sec:gen_arrays} extends these results
to general $k$ dimensional arrays where each 
dimension is independently exchangeable. To arrive
at this result we extend
the results of \citet{freedman1977remark}
to bounds on the total variation 
distance between sampling with and without replacement
from multiple urns.
The desire for an exact bound was posed as open
problem 15.10 by \citet[p.~137]{aldous1985exchangeability}.
In the Discussion we 
provide an explicit statement of a representation
theorem for distributions that are invariant
under the operations of a finite group $G$. 
This
representation can be used to compute a $G$-invariant
non-parametric maximum likelihood estimate of a
distribution as well as for the construction of 
test statistics for the infinite exchangeability of a distribution.

\section{Sharp total variation bounds 
for networks}\label{sec:tvbounds}
In this section we derive bounds for the total variation
distance between the distributions of
finitely and infinitely exchangeable networks. 
This quantitative summary provides insight into
the efficacy of the assumption of infinite
exchangeability. By taking limits of the bounds
in the number of nodes in the network we recover
the classical Aldous-Hoover representation.
Recall that a distribution of a network with
$m$ nodes is exchangeable, that is its representation
as a square relational data matrix $X$ 
is jointly row-column exchangeable, if
for all $g\in S_m(=\text{symmetric group})$ 
we have $\Pr(gX=A)=\Pr(X=A)$ where $g$ acts on the rows and columns of $X$ simultaneously. 
The results
of \cite{aldous1981representations} and \cite{hoover1979relations}
provide a representation theorem for 
infinitely exchangeable networks in the spirit 
of de Finetti's theorem for infinitely exchangeable
sequences: If $(X_{ij})_{1,1}^{\infty,\infty}$
is an infinitely exchangeable network then
it can be written as
$X_{ij}=f(\alpha,u_i,u_j,\lambda_{ij})$
for a measurable function $f$ that is symmetric
in its middle arguments and
$\alpha,u_i,\lambda_{ij}$
all independent ${\rm uniform}(0,1)$. This is
explicitly called joint or weak exchangeability.
This representation is frequently
used when describing relational data where the object of
interest is a square array $X$. 

It is natural to ask how close the finitely exchangeable 
distributions are to the Aldous-Hoover representation
of an infinitely row-column exchangeable array. 
The example of the failure of the Aldous-Hoover
representation for jointly exchangeable arrays is
similar in spirit to the classical example
provided by \citet{diaconis1977finite} for the failure
of the de Finetti's theorem for a finite sequence. 
We first present the sharp total variation bound for
finitely exchangeable sequences found in 
\citet{diaconis1980finite} and then employ a singular
value decomposition to extend the result
to networks.

\begin{thm}[\citet{diaconis1980finite} Theorem 13]\label{pd_thm}
Let $(\mathcal{S},\mathcal{B})$ be an abstract measurable
space and let $\mathcal{S}^\star$ be the set of probabilities
on $(\mathcal{S},\mathcal{B})$ endowed with the smallest
$\sigma$-field $\mathcal{B}^\star$ that makes $p\rightarrow p(A)$
measurable for all $A\in\mathcal{B}$. 
Let $P$ be an exchangeable probability on $(\mathcal{S}^n,\mathcal{B}^n)$, $P_k$ be the projection of $P$ onto $(\mathcal{S}^k,\mathcal{B}^k)$. 
There exists a probability $\mu$ on $(\mathcal{S}^\star,\mathcal{B}^\star)$ such that if $P_{\mu k}$ is a projection of an infinitely exchangeable distribution onto $(\mathcal{S}^k,\mathcal{B}^k)$ with mixing measure $\mu$, then
$\|P_k-P_{\mu k}\|\leq 2\beta(k,n)$ for all $k\leq n$ where $\beta(k,n)=1-n^{-k}n!/(n-k)!$.
\end{thm}
\begin{proof}[Proof outline]
This result is a direct consequence of finding the 
total variation bounds for sampling with and without
replacement from $\mathcal{S}^k$. Sampling without replacement represents an extreme point of the space of exchangeable distributions while sampling with replacement is a power probability. By letting $\mu$ be the distribution of the empirical measure implied by $P$, the difference $P_k(A)-P_{\mu k}(A)$ (for $A\in\mathcal{S}^k$) can be written as a mixture of differences between sampling without and with replacement. 
The sharpness
of the bound is immediate as it is achieved by the
distance between sampling with and without
replacement from an unconstrained statespace $\mathcal{S}$. 
\end{proof}
 
The projection onto $(\mathcal{S}^k,\mathcal{B}^k)$
describes the notion of extendibility of an exchangeable
distribution. That is, the theorem requires that
an exchangeable distribution on $(\mathcal{S}^k,\mathcal{B}^k)$
be extendible to 
an exchangeable distribution
on $(\mathcal{S}^n,\mathcal{B}^n)$ as such $P_k$ is a
marginal distribution of $P$. A similar construct is
available for networks as well: an exchangeable distribution on 
an $m$ dimensional network is extendible to $r$ dimensions
if it is the marginal of an $r$ dimensional exchangeable
distribution on networks. Infinite extendibility 
is equivalent to infinite exchangeability and so 
finite extendibility is a weaker assumption.

\begin{thm}[Total variation of jointly row-column exchangeable arrays]\label{thm:joint}
Finitely joint row-column ($m$ dimensional) exchangeable 
(and $r$ extendible) probabilities are bounded in total variation
distance from infinitely joint row-column exchangeable probabilities by 
$2\beta(m,r)$. This bound is sharp.
\end{thm}
\begin{proof}
We first demonstrate that the distribution of $X$ is jointly
row-column exchangeable if and only if the distribution of
its left and right eigenvectors is row exchangeable.
Since the singular
value decomposition is not unique, care must be taken.
First, we overload notation to write the singular value decomposition (SVD) of $X$ as 
$X\overset{d}{=}U(X)D(X)V(X)^t$. 
We choose a decomposition such that $D(X)$ is a diagonal matrix with non-negative entries in descending order.
The functions $U(\cdot)$ and $V(\cdot)$ map to a pre-specified choice of left and right singular vectors of $X$ (for example, for non-degenerate singular values one can use the SignFlip function of \cite{bro2008resolving}). For null singular values, the corresponding columns of $U(X)$ and $V(X)$ are distributed uniformly over the bases for the left and right null spaces of $X$.
Define a set $E(X)=\{O \ {\rm orthogonal}:D(X) = OD(X)O^t\}$, the space of orthogonal matrices that leaves the diagonal singular value matrix unchanged.
For example: If all the singular values are 
unique then $E(X)$ contains all diagonal matrices with entries $\pm 1$. If the 
first two singular value are the same then $E(X)$ also contains orthogonal matrices for which the leading $2\times 2$ 
submatrix is an orthogonal matrix.
Now, let the random variable $S(X)$ be uniformly distributed on $E(X)$ and define the random variables 
$$\tU(X)\overset{d}{=}U(X)S(X),\ \tV(X)\overset{d}{=}V(X)S(X).$$
These variables define random rotations of the left and right singular vectors of $X$. 
This mechanism of introducing additional 
randomness for identifying distributions of
eigenvectors was previously used in \citep{tao2011random}.
Since $S(X)$ is uniformly distributed on $E(X)$, it is clear that $\tU(X)$ and $\tV(X)$ are uniformly distributed on the equivalence classes
$${\rm eqc}(U(X))=\{U(X)O:O\in E(X)\},\ {\rm eqc}(V(X))=\{V(X)O:O\in E(X)\}.$$
Let $g$ be an element of the permutation group and define its action on $X$ as the permutation of the rows and columns jointly and its action on the singular vectors as the permutation of the rows only.
Since ${\rm eqc}(U(gX))=g{\rm eqc}(U(X))$ and ${\rm eqc}(V(gX))=g{\rm eqc}(V(X))$ we have that $gX\overset{d}{=}X$ implies $g[\tU(X);\tV(X)]\overset{d}{=}[\tU(X);\tV(X)]$.
We have
thus shown that if $X$ is row-column exchangeable then
$[\tU(X);\tV(X)]$ is row-exchangeable. 


If the right and left eigenvectors are row exchangeable
then $X$ is jointly row-column exchangeable by construction.
To apply Theorem \ref{pd_thm} we require extendibility of a sequence. To understand extendibility we note that if an $r\times r$ matrix
is row-column exchangeable then the an $m\times m$ submatrix is also row-column exchangeable. Defining the $U$ and $V$ above be the $m\times r$ rows of the singular value decomposition of the large $r\times r$ matrix the above results hold. Constructing the relevant $[\tU,\tV]$ matrix and treating the rows as a sequence we see that we can directly apply Theorem \ref{pd_thm} to 
get the desired result. 
\end{proof}

The theorem translates the discussion of the joint row-column
exchangeability of a distribution on square matrices to the discussion 
of the joint row exchangeability of orthogonal matrices.

\section{Sharp total variation bounds for general arrays}\label{sec:gen_arrays}
Without loss of generality, in this section we prove theorems 
for two dimensional arrays and then provide an outline
of how one extends the results to higher dimensions. 
Recall 
that a distribution of a $m\times n$ array $X$ is 
row-column exchangeable if for all
$(g_1,g_2)\in S_m\times S_n$ we have 
$\Pr((g_1,g_2)X=A)=\Pr(X=A)$ where $g_1$ acts on the rows and $g_2$ acts on the columns of $X$. 
The Aldous-Hoover representation
for an infinitely row-column exchangeable array
is: $(X_{ij})_{1,1}^{\infty,\infty}$ is infinitely row-column
exchangeable then $X_{ij}=f(\alpha,u_i,v_j,\lambda_{ij})$
for a measurable function $f$ and $\alpha,u_i,v_j,\lambda_{ij}$ 
all independent ${\rm uniform}(0,1)$ random variables. 
To find the total variation bounds for general exchangeable
arrays we require the following extension of Theorem~\ref{pd_thm}:

\begin{thm}[Total variation bounds for partial exchangeability]\label{thm:partial}
Let  $(\mathcal{S}_1,\mathcal{B}_1)$,  $(\mathcal{S}_2,\mathcal{B}_2)$
be abstract measurable
spaces and let 
$\mathcal{S}_i^\star$
be the set of probabilities
on $(\mathcal{S}_i,\mathcal{B}_i)$ endowed with the smallest
$\sigma$-field
$\mathcal{B}_i^\star$
that makes $p\rightarrow p(A)$
measurable for all $A\in\mathcal{B}_i$. 
Let $P$ be a partially exchangeable probability on
 $(\mathcal{S}_1^{n_1}\times\mathcal{S}_2^{n_2},\mathcal{B}_1^{n_1}\times\mathcal{B}_2^{n_2})$, $P_{k_1,k_2}$ 
 be the projection of $P$ onto
  $(\mathcal{S}_1^{k_1}\times\mathcal{S}_2^{k_2},\mathcal{B}_1^{k_1}\times\mathcal{B}_2^{k_2})$.
Then there exists a probability $\mu$ on 
$(\mathcal{S}^\star_1\times\mathcal{S}^\star_2,\mathcal{B}^\star_1\times\mathcal{B}^\star_2	)$
such that
if $P_{\mu k_1,k_2}$ is a projection of an infinitely
partially exchangeable distribution onto
 $(\mathcal{S}_1^{k_1}\times\mathcal{S}_2^{k_2},\mathcal{B}_1^{k_1}\times\mathcal{B}_2^{k_2})$ with mixing measure $\mu$, then
$$\|P_{k_1,k_2}-P_{\mu k_1,k_2}\|\leq 2\beta((k_1,k_2),(n_1,n_2))\ \forall k_1\leq n_1, k_2\leq n_2$$ where 
$$\beta((k_1,k_2),(n_1,n_2))=1-n_1^{-k_1}n_1!/(n_1-k_1)!n_2^{-k_2}n_2!/(n_2-k_2)!.$$
This bound is tight.
\end{thm}
\begin{proof}
The proof is the same as the proof of Theorem \ref{pd_thm}
after making the obvious substitutions. 
Specifically, in Theorem \ref{pd_thm} we employed results on sampling from an urn that included all possible sequences of length $n$ composed of elements of a set $\mathcal{S}$. Here we consider sampling from a pair of urns that together include all possible pairs of sequences of length $n_1+n_2$ where the first sequence is composed of elements of $\mathcal{S}_1$ while the second of elements of $\mathcal{S}_2$. 
The distributions we compare are sampling with and without replacement 
from this collection of urns. Any finitely partially exchangeable distributions	
will be a mixture of the distributions of sampling without replacement from these urns
while the projection of infinitely partially exchangeable distributions
will be a mixture of the distributions of sampling with replacement from
these urns. The bound and the tightness of the bound are given by
the technical lemmas in the Appendix.
\end{proof}

We now extend Theorem~\ref{thm:joint} to separately
exchangeable arrays. 
\begin{thm}[Total variation bounds]\label{thm:tv_bounds}
Let $P_{m,n}^{r,q}$ be a distribution on $m\times n$ dimensional 
matrices that is finitely 
row-column exchangeable and is extendible to a row-column
exchangeable distribution on $r\times q$ dimensional matrices.
Let $P_{\mu,m,n}$ be the projection of
an infinitely row-column exchangeable distribution onto the space
of $m\times n$ dimensional matrices. Then
for $\|\cdot\|$ representing the total variation norm we have
\begin{align*}
\|P_{m,n}^{r,q}-P_{\mu,m,n}\|\leq \beta((m,n),(r,q)).
\end{align*}
The bound is tight.
\end{thm}
\begin{proof}
The proof is similar to that of Theorem \ref{thm:joint}. The singular 
value decomposition of the matrix 
$X\overset{d}{=}U(X)D(X)V(X)^t$ is again considered, but the separate 
exchangeability of the rows and columns implies that the rows of the 
left and right eigenvector matrices are separately permuted. The random variables $S(X),\tU(X),\tV(X)$ are constructed as in Theorem \ref{thm:joint}.
By concatenating
$\tU(X)$ and $\tV(X)$ along the rows into $Z=[\tU(X)^t;\tV(X)^t]^t$, a $m+n$ row matrix,
we see that separate row column exchangeability of $X$ is equivalent
to the partial exchangeability of the rows of $Z$. The results then follows from
Theorem \ref{thm:partial}.
\end{proof}

The contribution of Theorem \ref{thm:tv_bounds} is two-fold. As far
as we know, this is the first time an explicit form for the total variation bound 
for row-column exchangeability has been proposed. 
More importantly, we now see that the rate of convergence of the total
variation bound is bounded by $(m^2-m)/q+(n^2-n)/r$ for matrices, 
similar to the results
of \citet{diaconis1980finite} for sequences. Knowing this bound
provides a justification for
 the assumption of 
infinite row-column exchangeability (allowing for an Aldous-Hoover 
representation) 
in high-dimensional statistics when both the subjects and features are
not independent but can assumed to be exchangeable and network analysis
where the labels of the nodes do not carry information. 

To extend the results to higher dimensions we require a generalization
of the singular value decomposition. For example,
we can consider the Higher Order Singular Value Decomposition of \citet{kolda2009tensor}: A $k$-dimensional array with
dimensions $n_1,\dots,n_k$ can be decomposed via the
Tucker decomposition into a core array $S$ and $k$ orthogonal
matrices $U_1,\dots,U_k$. 
A distribution over arrays that is invariant under 
independent permutations
of the index sets of the different dimensions of the array
is equivalent to a distribution over arrays that is 
invariant under independent
permutations of the rows of the orthogonal matrices $U_1,\dots,U_k$.
Lemma \ref{d_freed} of 
\ref{sec:appA} provides a 
$k$-dimensional version of the bound in Theorem \ref{thm:partial}
and the proof of the sharpness of the bound follows with the
obvious changes.

\section{Discussion}
In this paper we have demonstrated sharp bounds for 
the total variation distance between finitely and 
infinitely exchangeable distributions for networks
and for general arrays. These results answer an
open question of \citet{aldous1985exchangeability}. 
Addressing the original motivation behind the paper, 
the tight bound for the jointly row-column exchangeable 
distributions provides an insight into the efficacy of the 
assumption of 
existence of an Aldous-Hoover representation when analyzing 
relational data. 
Similar results for sequences were previously used to justify the use of de Finetti's
theorem for Bayesian causal inference \citep{rubin1978bayesian}.
The magnitude of the deviation of parameter estimates 
for the finite exchangeability-infinite exchangeability paradigm 
is a part of ongoing research by the authors. We conjecture
that any estimator that has a faster rate of convergence
than the total variation bounds should have similar properties
whether one assumes finite or infinite exchangeability. This is motivated empirically by the successful use of infinitely exchangeable methods when the true generating process cannot be infinitely exchangeable as is the case of Sampson's monastery \citep{sampson1969crisis} and high school classrooms \citep{hoff2013likelihoods}. Success is measured by the external validation of the results by subject matter scientists. 

The results of the paper require an assumption of extendibility
of the networks and arrays in question. While we believe this
to be a more reasonable assumption than infinite exchangeability,
it might fail to hold in some situations. When this is the case
inference that does not take into account the finite exchangeability
of the data generating distribution
is prone to the errors described in the introduction.
To accommodate this scenario we present a general characterization
theorem for distributions that are invariant under finite
group operations. 

\begin{thm}[Characterization Theorem]\label{charthm}
Let $G$ be a finite group acting on a finite state space $\mathcal{X}$. 
The collection of $G$-invariant measures on $\mathcal{X}$ is 
a convex set with extreme points $e_{G(x)}$ indexed by the 
orbits of $\mathcal{X}$ and placing weight $1/|G(x)|$ on each member of
$G(x)$ and 0 everywhere else. 
\end{thm} 
\begin{proof}
Convexity is a trivial consequence of the convexity of
the general space of probability distributions on $\mathcal{X}$
when $|\mathcal{X}|<\infty$. First, if a distribution is $G$-invariant then the probability of two elements within an orbit must be equal and so it is constant on the orbits. Conversely, if a distribution is uniform over a single orbit then it is necessarily invariant under the group operation. To see that these $e_{G(x)}$ are extreme we note that they all have disjoint support by definition and so cannot be written as convex combinations of each other.
\end{proof}

General versions of this theorem are available 
in \citet{wijsman1957random} and \citet{eaton1989group}. 
Using this theorem we can construct 
a finite group invariance restricted empirical 
distribution. The estimator is essentially the same as that for
 the unrestricted empirical distribution but it 
places equal weight on each element of an orbit during estimation. 
Thus a single observation of $x$ would lead to 
an estimate that is uniform on the orbit of $x$.
This guarantees that the estimate is invariant itself.
This characterization also motivates a test for 
finite exchangeability: A test
based on $T=\|\mathbb{P}-\mathbb{P}_G\|$ where $\mathbb{P}$
is the unrestricted empirical distribution, 
$\mathbb{P}_G$  is the finite group invariance restricted 
empirical distribution, and $\|\cdot\|$ is total variation distance
rejects the null for large values of $T$. Further, a 
test for infinite exchangeability can also be constructed.
Specifically, 
one can test the null of infinite extendibility versus an alternative of
finite extendibility of a finitely exchangeable distribution
using the test statistic
$T=\max_{H}\{|H|:\exists P_{G|H}\ {\rm st}\ \mathbb{P}_G=P_{G|H}\}$
where $\mathbb{P}_G$ is the finitely exchangeable restricted empirical
distribution and $P_{G|H}$ is a distribution that is invariant
under both the larger group $H$ and the smaller group $G$. The 
test rejects for small $T$. 
For both tests the null distributions can be constructed
via Monte Carlo simulation.

\vspace{-.5cm}
\section*{Acknowledgments}
\noindent We would like to thank the editor, associate editor and two reviewers for detailed comments that have helped improve this manuscript. This work was supported, in part, by ONR award N00014-14-1-0485, by ARO MURI award W911NF-11-1-0036 and NSF CAREER grant IIS-1149662 to Harvard University. AV is an NSF MSPRF on NSF DMS-1402235. EMA is a Sloan Research Fellow and a Shutzer Fellow at the Radcliffe Institute for Advanced Studies.
\appendix
\appendix
\section{Technical lemmas}\label{sec:appA}
To prove Theorem \ref{thm:partial} we require
an extension of the results of \cite{freedman1977remark}
to problems with multiple urns. We follow Freedman's
notation and method. In particular, 
consider a collection of $d$ urns
$F_1,\dots,F_d$ where $F_i$ contains
elements $\{f_{i1},\dots,f_{in_i}\}$. Let
$F_i^{k_i}$ consist of the $k_i$-tuples
of elements of $F_i$. Sampling with 
replacement from the collection of $d$
urns induces a uniform probability
$M$ on $F_1^{k_1}\times \cdots\times F_d^{k_d}$
where 
$$M(\{s_{11},\dots,s_{1k_1}\},\dots,\{s_{d1},\dots,s_{dk_d}\})=\frac{1}{n_1^{k_1}\cdots n_d^{k_d}}$$
Letting $B$ be the vectors 
$\{s_{11},\dots,s_{1k_1}\},\dots,\{s_{d1},\dots,s_{dk_d}\}\in F_1^{k_1}\times \cdots\times F_d^{k_d}$
for which all the components are unequal, sampling without 
replacement is given by
$$Q(\{s_{11},\dots,s_{1k_1}\},\dots,\{s_{d1},\dots,s_{dk_d}\})=\frac{(n_1-k_1)!\cdots(n_d-k_d)!}{n_1!\cdots n_d!}.$$

\begin{lem}[Sampling with and without replacement]\label{d_freed}
The total variation distance between sampling with 
and without replacement from multiple urns is 
 $$\|M-Q\|=1-\frac{n_1!\cdots n_d!}{(n_1-k_1)!\cdots(n_d-k_d)!n_1^{k_1}\cdots n_d^{k_d}}.	$$
\end{lem}
\begin{proof}
The proof is by construction. We note that $Q(B)=1$ and  
the total variation norm is given by $Q(B)-M(B)$. To calculate
$M(B)$ we note that 	the probability of any particular
sequence is equal under $M$ and so we simply need
to count the number of sequences in $B$. This is an 
easy combinatorial exercise and each urn has 
$\frac{n_i!}{(n_i-k_i)!}$ sequences of length $k_i$
with unique entries. Thus 
$M(B)=\frac{n_1!\cdots n_d!}{(n_1-k_1)!\cdots(n_d-k_d)!n_1^{k_1}\cdots n_d^{k_d}}$
and the lemma follows.
\end{proof}

\begin{lem}[Tight bounds]
The bound in Theorem \ref{thm:partial} is tight. That is 
$$\|P_{n_1,n_2,k_1,k_2}-P_{\mu,k_1,k_2}\|\geq\|P_{n_1,n_2,k_1,k_2}-M\|=1-\frac{n_1!n_2!}{(n_1-k_1)!(n_2-k_2)!n_1^{k_1} n_2^{k_2}}$$
where $M$ is as in Lemma \ref{d_freed}.
\end{lem}
\begin{proof}
Without loss of generality we let the distributions
of interest be over two urns with balls labeled $1^1,\dots,n_1^1$
and $1^2,\dots,n_2^2$ respectively. 
We note that the equality in the statement is given by Lemma \ref{d_freed}. Thus
as in the original proof of tightness in \cite{diaconis1980finite}
we must show that for $B$ as defined in Lemma \ref{d_freed} we 
have $(p_1^{k_1}p_2^{k_2})(B)\leq M(B)$ for any distributions
$p_1,p_2$ on the the sets $\{1^1,\dots,n_1^1\}$ and $\{1^2,\dots,n_2^2\}$ respectively. 
To see that this is true we note
$M$ is a products of two pure power probabilities
and that we can write the set $B=B_1\times B_2$ where $B_1$ are all the sets of
length $k_1$ of unique elements in the first urn and similarly for $B_2$. 
As in \cite{diaconis1980finite} we note that the Schur
 convexity of the indicator functions $1_{B_1}$ and $1_{B_2}$ 
 implies that $p_i^{k_i}(B_i)\leq M_i(B_i)$. This implies that 
 $P_{\mu,k_1,k_2}(B)\leq M(B)$ as desired.
\end{proof}
%

\vspace{-0.5cm}
\bibliographystyle{apalike}
\bibliography{biblio}
\end{document}